\documentclass[11pt,reqno]{amsart}
\usepackage{xr}
\usepackage{hyperxmp}
\usepackage[shortlabels]{enumitem}      
\usepackage[type={CC},modifier={by-nc-nd},version={4.0},imagewidth=2.5cm]{doclicense}

\usepackage{booktabs}
\usepackage{amssymb}
\usepackage{graphicx}
\usepackage{epstopdf}
\usepackage{url}
\usepackage{color}
\usepackage[usenames,dvipsnames,svgnames,table]{xcolor}
\usepackage[square,numbers,sort&compress]{natbib}

\usepackage{tikz}

\usepackage{pgfplots} 

\usetikzlibrary{positioning}

\def\visible<#1>{}  

\usepackage[utf8]{inputenc}

\usepackage[english]{babel}
\usepackage{amsfonts}
\usepackage{amsmath}
\usepackage{latexsym}
\usepackage{color}
\usepackage{subfigure}

\usepackage{hyperref}  

\usepackage{ifpdf}

\graphicspath{{figures/}}

\DeclareMathOperator    \intr                   {int}

\newcommand{\old}[1]{{}}

\newcommand{\bb}{\mathbb}

\newcommand{\R}{\bb R}

\newcommand{\Z}{\bb Z}

\newcommand{\setcond}[2]{\left\{\, #1 : #2 \,\right\}}

\renewcommand{\P}{\mathcal{P}}

\newenvironment{psmallmatrixbig}{\bigl(\smallmatrix}{\endsmallmatrix\bigr)}
\newcommand\InlineFrac[2]{#1/#2}  
\newcommand\ColVec[3][\relax]
{
  \ifx#1\relax
  \bgroup\let\frac=\InlineFrac\begin{psmallmatrixbig}#2\vphantom{/}\\#3\vphantom{/}\end{psmallmatrixbig}\egroup
  \else
  \bgroup\let\frac=\InlineFrac\begin{psmallmatrixbig}\ifx#200\else#2/#1\fi\\\ifx#300\else#3/#1\fi\end{psmallmatrixbig}\egroup
  \fi
}

\newtheorem{theorem}{Theorem}[section]

\makeatletter
\newcommand\MkNewTheorem[2]{%
  \newtheorem{#1}{#2}[section]
  \expandafter\def\csname c@#1\endcsname{\c@theorem}
  \expandafter\def\csname p@#1\endcsname{\p@theorem}
  \expandafter\def\csname the#1\endcsname{\thetheorem}
  \expandafter\def\csname #1name\endcsname{#2}
}
\makeatother

\MkNewTheorem{corollary}{Corollary}
\MkNewTheorem{lemma}{Lemma}
\MkNewTheorem{proposition}{Proposition}
\MkNewTheorem{prop}{Proposition}
\MkNewTheorem{claim}{Claim}
\MkNewTheorem{observation}{Observation}
\MkNewTheorem{obs}{Observation}
\MkNewTheorem{conjecture}{Conjecture}
\MkNewTheorem{openquestion}{Open question}

\theoremstyle{definition}
\MkNewTheorem{example}{Example}
\MkNewTheorem{exercise}{Exercise}
\MkNewTheorem{notation}{Notation}
\MkNewTheorem{assumption}{Assumption}
\MkNewTheorem{definition}{Definition}
\MkNewTheorem{remark}{Remark}
\MkNewTheorem{goal}{Goal}

\makeatletter
\let\OurMathBbAux=\mathbb
\DeclareRobustCommand\OurMathBb{\OurMathBbAux}
\let\mathbb=\OurMathBb
\let\bfseries=\undefined
\DeclareRobustCommand\bfseries
{\not@math@alphabet\bfseries\mathbf
  \boldmath\fontseries\bfdefault\selectfont\let\OurMathBbAux=\mathbf}
\def\@thm#1#2#3{%
  \ifhmode\unskip\unskip\par\fi
  \normalfont
  \trivlist
  \let\thmheadnl\relax
  \let\thm@swap\@gobble
  \thm@notefont{\fontseries\mddefault\upshape\unboldmath}
  \thm@headpunct{.}
  \thm@headsep 5\p@ plus\p@ minus\p@\relax
  \thm@space@setup
  #1
  \@topsep \thm@preskip               
  \@topsepadd \thm@postskip           
  \def\@tempa{#2}\ifx\@empty\@tempa
    \def\@tempa{\@oparg{\@begintheorem{#3}{}}[]}%
  \else
    \refstepcounter{#2}%
    \def\@tempa{\@oparg{\@begintheorem{#3}{\csname the#2\endcsname}}[]}%
  \fi
  \@tempa
}
\makeatother

\makeatletter
\renewcommand{\pod}[1]
{\allowbreak\mathchoice{\mkern18mu}{\mkern8mu}{\mkern8mu}{\mkern8mu}(#1)}
\makeatother

\chardef\Myunderscore=`\_
\newcommand\underscore{\Myunderscore\allowbreak}

\newcommand\githubsearchurl{https://github.com/mkoeppe/infinite-group-relaxation-code/search}
\pgfkeyssetvalue{/sagefunc/california_ip}{\href{\githubsearchurl?q=\%22def+california_ip(\%22}{\sage{california\underscore{}ip}}}
\pgfkeyssetvalue{/sagefunc/automorphism}{\href{\githubsearchurl?q=\%22def+automorphism(\%22}{\sage{automorphism}}}
\pgfkeyssetvalue{/sagefunc/multiplicative_homomorphism}{\href{\githubsearchurl?q=\%22def+multiplicative_homomorphism(\%22}{\sage{multiplicative\underscore{}homomorphism}}}
\pgfkeyssetvalue{/sagefunc/projected_sequential_merge}{\href{\githubsearchurl?q=\%22def+projected_sequential_merge(\%22}{\sage{projected\underscore{}sequential\underscore{}merge}}}
\pgfkeyssetvalue{/sagefunc/restrict_to_finite_group}{\href{\githubsearchurl?q=\%22def+restrict_to_finite_group(\%22}{\sage{restrict\underscore{}to\underscore{}finite\underscore{}group}}}
\pgfkeyssetvalue{/sagefunc/restrict_to_finite_group_3}{\href{\githubsearchurl?q=\%22def+restrict_to_finite_group_3(\%22}{\sage{restrict\underscore{}to\underscore{}finite\underscore{}group\underscore{}3}}}
\pgfkeyssetvalue{/sagefunc/interpolate_to_infinite_group}{\href{\githubsearchurl?q=\%22def+interpolate_to_infinite_group(\%22}{\sage{interpolate\underscore{}to\underscore{}infinite\underscore{}group}}}
\pgfkeyssetvalue{/sagefunc/two_slope_fill_in}{\href{\githubsearchurl?q=\%22def+two_slope_fill_in(\%22}{\sage{two\underscore{}slope\underscore{}fill\underscore{}in}}}
\pgfkeyssetvalue{/sagefunc/ll_strong_fractional}{\href{\githubsearchurl?q=\%22def+ll_strong_fractional(\%22}{\sage{ll\underscore{}strong\underscore{}fractional}}}
\pgfkeyssetvalue{/sagefunc/hildebrand_2_sided_discont_1_slope_1}{\href{\githubsearchurl?q=\%22def+hildebrand_2_sided_discont_1_slope_1(\%22}{\sage{hildebrand\underscore{}2\underscore{}sided\underscore{}discont\underscore{}1\underscore{}slope\underscore{}1}}}
\pgfkeyssetvalue{/sagefunc/hildebrand_2_sided_discont_2_slope_1}{\href{\githubsearchurl?q=\%22def+hildebrand_2_sided_discont_2_slope_1(\%22}{\sage{hildebrand\underscore{}2\underscore{}sided\underscore{}discont\underscore{}2\underscore{}slope\underscore{}1}}}
\pgfkeyssetvalue{/sagefunc/hildebrand_discont_3_slope_1}{\href{\githubsearchurl?q=\%22def+hildebrand_discont_3_slope_1(\%22}{\sage{hildebrand\underscore{}discont\underscore{}3\underscore{}slope\underscore{}1}}}
\pgfkeyssetvalue{/sagefunc/dr_projected_sequential_merge_3_slope}{\href{\githubsearchurl?q=\%22def+dr_projected_sequential_merge_3_slope(\%22}{\sage{dr\underscore{}projected\underscore{}sequential\underscore{}merge\underscore{}3\underscore{}slope}}}
\pgfkeyssetvalue{/sagefunc/chen_4_slope}{\href{\githubsearchurl?q=\%22def+chen_4_slope(\%22}{\sage{chen\underscore{}4\underscore{}slope}}}
\pgfkeyssetvalue{/sagefunc/gmic}{\href{\githubsearchurl?q=\%22def+gmic(\%22}{\sage{gmic}}}
\pgfkeyssetvalue{/sagefunc/gj_2_slope}{\href{\githubsearchurl?q=\%22def+gj_2_slope(\%22}{\sage{gj\underscore{}2\underscore{}slope}}}
\pgfkeyssetvalue{/sagefunc/gj_2_slope_repeat}{\href{\githubsearchurl?q=\%22def+gj_2_slope_repeat(\%22}{\sage{gj\underscore{}2\underscore{}slope\underscore{}repeat}}}
\pgfkeyssetvalue{/sagefunc/dg_2_step_mir}{\href{\githubsearchurl?q=\%22def+dg_2_step_mir(\%22}{\sage{dg\underscore{}2\underscore{}step\underscore{}mir}}}
\pgfkeyssetvalue{/sagefunc/kf_n_step_mir}{\href{\githubsearchurl?q=\%22def+kf_n_step_mir(\%22}{\sage{kf\underscore{}n\underscore{}step\underscore{}mir}}}
\pgfkeyssetvalue{/sagefunc/gj_forward_3_slope}{\href{\githubsearchurl?q=\%22def+gj_forward_3_slope(\%22}{\sage{gj\underscore{}forward\underscore{}3\underscore{}slope}}}
\pgfkeyssetvalue{/sagefunc/drlm_backward_3_slope}{\href{\githubsearchurl?q=\%22def+drlm_backward_3_slope(\%22}{\sage{drlm\underscore{}backward\underscore{}3\underscore{}slope}}}
\pgfkeyssetvalue{/sagefunc/drlm_2_slope_limit}{\href{\githubsearchurl?q=\%22def+drlm_2_slope_limit(\%22}{\sage{drlm\underscore{}2\underscore{}slope\underscore{}limit}}}
\pgfkeyssetvalue{/sagefunc/drlm_2_slope_limit_1_1}{\href{\githubsearchurl?q=\%22def+drlm_2_slope_limit_1_1(\%22}{\sage{drlm\underscore{}2\underscore{}slope\underscore{}limit\underscore{}1\underscore{}1}}}
\pgfkeyssetvalue{/sagefunc/bhk_irrational}{\href{\githubsearchurl?q=\%22def+bhk_irrational(\%22}{\sage{bhk\underscore{}irrational}}}
\pgfkeyssetvalue{/sagefunc/bccz_counterexample}{\href{\githubsearchurl?q=\%22def+bccz_counterexample(\%22}{\sage{bccz\underscore{}counterexample}}}
\pgfkeyssetvalue{/sagefunc/drlm_3_slope_limit}{\href{\githubsearchurl?q=\%22def+drlm_3_slope_limit(\%22}{\sage{drlm\underscore{}3\underscore{}slope\underscore{}limit}}}
\pgfkeyssetvalue{/sagefunc/dg_2_step_mir_limit}{\href{\githubsearchurl?q=\%22def+dg_2_step_mir_limit(\%22}{\sage{dg\underscore{}2\underscore{}step\underscore{}mir\underscore{}limit}}}
\pgfkeyssetvalue{/sagefunc/rlm_dpl1_extreme_3a}{\href{\githubsearchurl?q=\%22def+rlm_dpl1_extreme_3a(\%22}{\sage{rlm\underscore{}dpl1\underscore{}extreme\underscore{}3a}}}
\pgfkeyssetvalue{/sagefunc/hildebrand_5_slope_22_1}{\href{\githubsearchurl?q=\%22def+hildebrand_5_slope_22_1(\%22}{\sage{hildebrand\underscore{}5\underscore{}slope\underscore{}22\underscore{}1}}}
\pgfkeyssetvalue{/sagefunc/hildebrand_5_slope_24_1}{\href{\githubsearchurl?q=\%22def+hildebrand_5_slope_24_1(\%22}{\sage{hildebrand\underscore{}5\underscore{}slope\underscore{}24\underscore{}1}}}
\pgfkeyssetvalue{/sagefunc/hildebrand_5_slope_28_1}{\href{\githubsearchurl?q=\%22def+hildebrand_5_slope_28_1(\%22}{\sage{hildebrand\underscore{}5\underscore{}slope\underscore{}28\underscore{}1}}}
\pgfkeyssetvalue{/sagefunc/extremality_test}{\href{\githubsearchurl?q=\%22def+extremality_test(\%22}{\sage{extremality\underscore{}test}}}
\pgfkeyssetvalue{/sagefunc/plot_2d_diagram}{\href{\githubsearchurl?q=\%22def+plot_2d_diagram(\%22}{\sage{plot\underscore{}2d\underscore{}diagram}}}
\pgfkeyssetvalue{/sagefunc/generate_example_e_for_psi_n}{\href{\githubsearchurl?q=\%22def+generate_example_e_for_psi_n(\%22}{\sage{generate\underscore{}example\underscore{}e\underscore{}for\underscore{}psi\underscore{}n}}}
\pgfkeyssetvalue{/sagefunc/chen_3_slope_not_extreme}{\href{\githubsearchurl?q=\%22def+chen_3_slope_not_extreme(\%22}{\sage{chen\underscore{}3\underscore{}slope\underscore{}not\underscore{}extreme}}}
\pgfkeyssetvalue{/sagefunc/psi_n_in_bccz_counterexample_construction}{\href{\githubsearchurl?q=\%22def+psi_n_in_bccz_counterexample_construction(\%22}{\sage{psi\underscore{}n\underscore{}in\underscore{}bccz\underscore{}counterexample\underscore{}construction}}}
\pgfkeyssetvalue{/sagefunc/gomory_fractional}{\href{\githubsearchurl?q=\%22def+gomory_fractional(\%22}{\sage{gomory\underscore{}fractional}}}
\pgfkeyssetvalue{/sagefunc/not_extreme_1}{\href{\githubsearchurl?q=\%22def+not_extreme_1(\%22}{\sage{not\underscore{}extreme\underscore{}1}}}
\pgfkeyssetvalue{/sagefunc/bhk_irrational_extreme_limit_to_rational_nonextreme}{\href{\githubsearchurl?q=\%22def+bhk_irrational_extreme_limit_to_rational_nonextreme(\%22}{\sage{bhk\underscore{}irrational\underscore{}extreme\underscore{}limit\underscore{}to\underscore{}rational\underscore{}nonextreme}}}
\pgfkeyssetvalue{/sagefunc/drlm_gj_2_slope_extreme_limit_to_nonextreme}{\href{\githubsearchurl?q=\%22def+drlm_gj_2_slope_extreme_limit_to_nonextreme(\%22}{\sage{drlm\underscore{}gj\underscore{}2\underscore{}slope\underscore{}extreme\underscore{}limit\underscore{}to\underscore{}nonextreme}}}
\pgfkeyssetvalue{/sagefunc/not_minimal_2}{\href{\githubsearchurl?q=\%22def+not_minimal_2(\%22}{\sage{not\underscore{}minimal\underscore{}2}}}
\pgfkeyssetvalue{/sagefunc/drlm_not_extreme_1}{\href{\githubsearchurl?q=\%22def+drlm_not_extreme_1(\%22}{\sage{drlm\underscore{}not\underscore{}extreme\underscore{}1}}}

\DeclareRobustCommand\sage[1]{\texttt{#1}}
\DeclareRobustCommand\sagefunc[1]{\pgfkeys{/sagefunc/#1}}

\usepackage[normalem]{ulem}

\usepackage{mk-pdf-copypastable}
\usepackage{verbatim}

  \newcommand\CompendiumGraphics[1]{\includegraphics[width=.8\linewidth]{#1}}
  \newcommand\CompendiumGridEntry[1]{\CompendiumGraphics{#1}\par{\tiny\sagefunc{#1}\par}\vspace*{-3ex}}

\graphicspath{{}{../survey/}}

\title[An electronic compendium of extreme functions]{An electronic compendium\\
  of extreme functions
  for the 
  \\
  Gomory--Johnson
  infinite group problem}

\thanks{The authors gratefully acknowledge partial support from the National Science
  Foundation through grant DMS-1320051 awarded to M.~K\"oppe.}

\author{Matthias K\"oppe}
\address{Matthias K\"oppe: Dept.\ of Mathematics, University of California, Davis}
\email{mkoeppe@math.ucdavis.edu}

\author{Yuan Zhou} 
\address{Yuan Zhou: Dept.\ of Mathematics, University of California, Davis}
\email{yzh@math.ucdavis.edu}

\date{2015-05-21, \textit{Revision}: 1810}

\begin{document}

\begin{abstract}
  In this note we announce the availability of an electronic compendium of
  extreme functions for Gomory--Johnson's infinite group problem.  These
  functions serve as the strongest cut-generating functions for integer linear
  optimization problems.  We also close several gaps in the literature.
\end{abstract}

\maketitle
\insert\footins{
  \normalfont\footnotesize
  \interlinepenalty\interfootnotelinepenalty
  \splittopskip\footnotesep \splitmaxdepth \dp\strutbox
  \floatingpenalty10000 \hsize\columnwidth
  This is a post-print (accepted manuscript) version of this paper, 
  which has been published in Operations Research Letters \textbf{43} (2015),
  no.~4, 438--444, and is available from \url{https://doi.org/10.1016/j.orl.2015.06.004}.
  \doclicenseThis\par}


\section{Introduction}

The infinite group problem was introduced 42 years ago by Ralph Gomory and
Ellis Johnson in their groundbreaking papers titled \emph{Some continuous
  functions related to corner polyhedra I, II} \cite{infinite,infinite2}.  The
technique, investigating strong relaxations of integer linear programs by
convexity in a function space, has often been dismissed as ``esoteric.''
Now we recognize the infinite group problem as a technique which was
decades ahead of its time, and which may be the key to today's
pressing need for stronger, multi-row cutting plane approaches.  

In this note, however, we restrict ourselves to the single-row (or,
``one-dimensional'') infinite group problem, which has attracted most of the
attention in the past.  It can be written as
\begin{equation}
  \label{GP} 
  \begin{aligned}
    &\sum_{r \in \R} r\, y(r) \equiv f \pmod{1}, \\
    &y\colon \R\to\Z_+ \text{ is a function of finite support}, 
  \end{aligned}
\end{equation}
where $f$ is a given element of $\R\setminus \Z$. 
We study the convex hull $R_{f}(\R,\Z)$ of
the set of all functions $y\colon \R \to \Z_+$ satisfying the constraints
in~\eqref{GP}.  The elements of the convex hull are understood as functions
$y\colon \R \to \R_+$. 

After a normalization, valid inequalities for the convex set $R_{f}(\R,\Z)$
can be described using so-called \emph{valid functions} $\pi\colon \R\to\R$
via $$\langle \pi, y \rangle := \sum_{r \in \R} \pi(r)y(r) \geq 1.$$  In the
finite-dimensional case, instead of merely valid inequalities, one is
interested in stronger inequalities such as tight valid inequalities and
facet-defining inequalities.  These r\^oles are taken in our
infinite-dimensional setting by \emph{minimal functions} and \emph{extreme
  functions}.  Minimal functions are those valid functions that are pointwise
minimal; extreme functions are those that are not a proper convex combination
of other valid functions. 

By a theorem of Gomory and
Johnson~\cite{infinite}, minimal functions for $R_f(\R,\Z)$ are classified:  
They are exactly the subadditive functions $\pi\colon \R\to\R_+$ that are
periodic modulo~$1$ and satisfy the \emph{symmetry condition} $\pi(x) +
\pi(f - x) = 1$ for all $x\in\R$.  
A major goal of research is to obtain a
classification, or at least improved understanding, of the extreme functions as
well. 

We refer the interested reader to the recent surveys
\cite{corner_survey,Richard-Dey-2010:50-year-survey} and the forthcoming
survey \cite{igp_survey} for a more detailed exposition.

\medbreak

The main purpose of this note is to announce the availability of an Electronic
Compendium of Extreme Functions \cite{electronic-compendium}, implemented in
Python within the framework of the open-source computer algebra system Sage \cite{sage}.
We hope that it will facilitate new research on the infinite group problem, in
particular by enabling computational experiments, including those that would
investigate the strength of these functions to generate the coefficients of
cutting planes in a branch and cut algorithm. To our knowledge, such experiments have not
been conducted systematically with all known families of extreme functions.

The survey \cite{Richard-Dey-2010:50-year-survey} provided an excellent
service by organizing the extreme functions known at that time; yet our
knowledge has grown since then, and our Electronic Compendium provides for the
first time convenient and up-to-date access to the definitions of the
functions, bringing light to those functions hidden in obscurity.

The extremality proof for a given class of piecewise minimal valid functions
follows a standard pattern, which we illustrate with the proof of extremality
for a class of functions that was described in the literature but whose
extremality was unknown before (\autoref{s:dpl1_extreme}).  

This standard proof pattern can actually be fully automated for a given
piecewise minimal valid function with rational
data~\cite{basu-hildebrand-koeppe:equivariant}; see also
\cite[\autoref{survey:sec:one-two-dim}]{igp_survey}.  The Compendium is
released as part of software \cite{infinite-group-relaxation-code} that
implements this automated extremality test and thus enables computational
experiments to find new extreme functions, to make conjectures, and to verify
claims in the literature.  As an example regarding the latter, we highlight
the case of a family of piecewise linear functions with $3$~slopes studied by
Chen~\cite{chen}.  Chen's extremality proof is flawed, and using our software we
can easily determine the non-extremality of these functions
(\autoref{s:chen_3_slope_not_extreme}).

All definitions of extreme functions in the Compendium are provided with unit
tests using this extremality-testing algorithm, which help to ensure the
correctness of the compendium as we continue to add newly discovered functions
or generalize their constructions.

Some classes of piecewise linear minimal valid functions have historically
been proved extreme using a different type of proof, either by the connection
to a sequence of finite group problems, or a lemma regarding limits of
sequences of extreme functions.  We may now consider these proofs obsolete, as
they can be replaced by proofs following the standard pattern.  We illustrate
this by an example in~\autoref{s:new_proof_drlm_backward_3_slope}.

The only example known in the literature whose extremality proof does not
follow the standard pattern and cannot be automated is a family of non--piecewise
linear, measurable functions introduced by Basu et al.~\cite{bccz08222222} as a
counterexample to a conjecture by Gomory--Johnson.  Basu et al.'s functions
are uniform limits of certain 2-slope functions, which we identify as
\sagefunc{kf_n_step_mir}\footnote{Throughout this paper, we refer to an extreme function
  or a family of extreme functions by the name of the Sage function that
  constructs them; these names are shown in typewriter font.} 
functions.  The limit functions can be seen as
absolutely continuous, measurable, non--piecewise linear generalizations of 2-slope functions. 
In \autoref{s:bccz_counterexample}, we contribute a discussion of a limiting case of the same construction, in
which one of the two slopes disappears, giving a continuous (but not
absolutely continuous), measurable, non--piecewise 
linear ``1-slope function.'' It can be seen as a periodic, subadditive
version of Cantor's devil function.

At the end of the article (\autoref{s:usage}) we explain how to use the
Electronic Compendium and show an overview of the functions in it at the time
of writing of this article (\autoref{tab:compendium}).

\section{Extremality of Richard et al.'s $\mathrm{DPL}_1$-extreme function
  \sage{rlm\underscore{}dpl1\underscore{}extreme\underscore{}3a}}
\label{s:dpl1_extreme}

Richard et al.~\cite{Richard-Li-Miller-2009:Approximate-Liftings} study
families of piecewise linear minimal valid functions with a prescribed set of
breakpoints.\footnote{They present these results within their theory of
  approximate lifting, where pseudo-periodic superadditive functions appear.
  However, these can be transformed, via
  \cite[Proposition~18]{Richard-Li-Miller-2009:Approximate-Liftings}, to the
  standard 
  setting of minimal valid functions for the infinite group problem, which are
  periodic and subadditive.  We describe all functions after applying this
  transformation.}  Specifically, $\mathrm{DPL}_n$ functions are possibly
discontinuous piecewise linear functions with $2n+1$ subintervals in $[0,1]$, 
whose graphs include the closed line segment from $\ColVec{0}{0}$ to $\ColVec{f}{1}$.  They
are parametrized by a finite number of real values: the slopes of the
subintervals (whose lengths are fixed in advance) and the values of the jumps
at the (fixed) breakpoints.  Subadditivity of these functions is then
characterized by finitely many linear inequalities \cite[\autoref{survey:section:minimalityTest}]{igp_survey}.  The functions
corresponding to the extreme points of the polytope~$P\Theta_n$ are called 
``$\mathrm{DPL}_n$-extreme functions''.  These functions are candidates for,
but not guaranteed to be extreme functions for the infinite group problem,
because perturbations may exist that are not piecewise linear on the same set
of breakpoints but rather on a refinement; see
\cite[\autoref{survey:s:interpolation}]{igp_survey}. 

For small $n$, the extreme points of~$P\Theta_n$ and thus the resulting
$\mathrm{DPL}_n$-extreme functions can be listed.  Richard et
al.~\cite[Theorem 28]{Richard-Li-Miller-2009:Approximate-Liftings} accomplish
this for $n=1$, giving piecewise formulas for the extreme points in terms of
the single remaining breakpoint parameter~$f$.  For any value of~$f$, there are three
extreme points.  Richard et
al.~\cite{Richard-Li-Miller-2009:Approximate-Liftings} note that the first one
corresponds to the \sagefunc{gmic} function and the second one recovers
functions proved extreme in~\cite{dey1}, 
namely, \sage{\sagefunc{drlm_3_slope_limit}(f)} for $0 < f < \frac{1}{2}$ and
\sage{\sagefunc{drlm_2_slope_limit}(f, 1, 1)} for $\frac{1}{2} \leq f < 1$. We
note that for $\frac{1}{3} \leq f < 1$, the $\mathrm{DPL}_n$-extreme
function corresponding to the third extreme point of~$P\Theta_1$ is equal to
\sage{\sagefunc{drlm_2_slope_limit}(f, 1, 2)} and thus is extreme. We now
prove that for the other case where $0 < f < \frac{1}{3}$ the function is
extreme as well. 



\begin{figure}[t]
  \centering
  \includegraphics[width=0.8\linewidth]{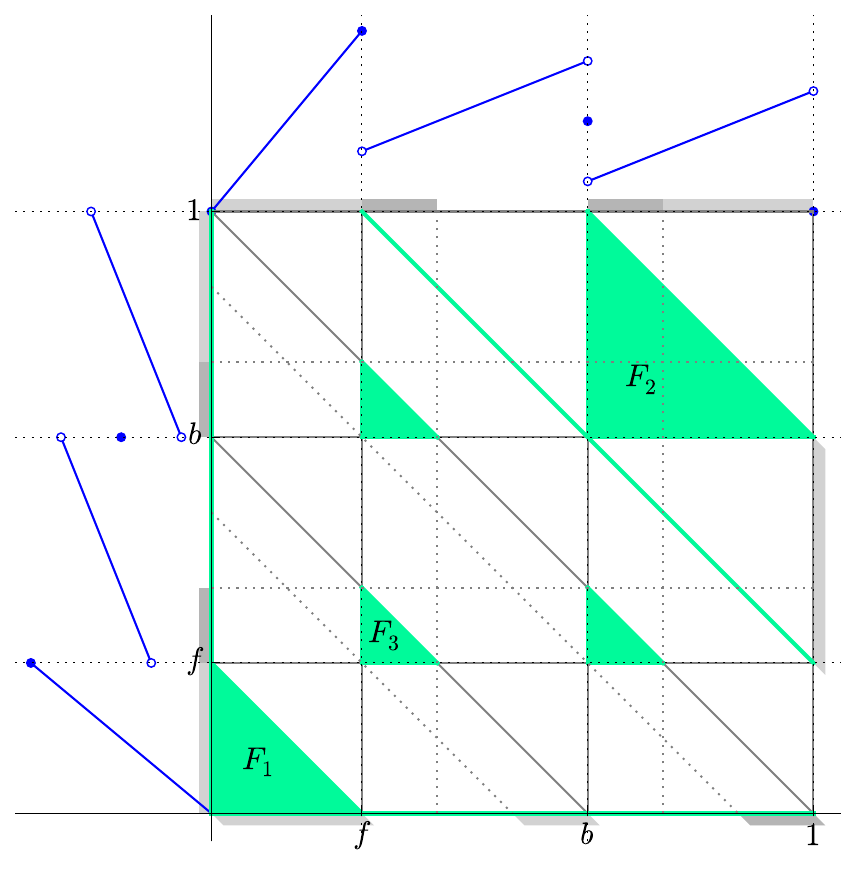}
  \caption{Diagram of the function \sagefunc{rlm_dpl1_extreme_3a} (\emph{blue
      graphs on the top and the left}) 
    and its polyhedral complex $\Delta\P$ (\emph{gray
      solid lines}).
    The set $E(\pi)$ is the union of the faces shaded in green.
    The \emph{heavy diagonal green line} $x + y = 1+f$ 
    corresponds to the symmetry condition (the line $x+y = f$ appears as an
    edge of $F_1$).  Vertices of $\Delta\P$ do not
    necessarily project (\emph{dotted gray lines}) to breakpoints.  At the
    borders, the projections $p_i(F)$ of two-dimensional additive faces are
    shown as \emph{gray shadows}: $p_1(F)$ at the top border, $p_2(F)$ at the
    left border, $p_3(F)$ at the bottom and the right borders.}
  \label{fig:rlm_dpl1_extreme_3a}
\end{figure}

\begin{theorem} \label{thm:rlm_dpl1_extreme_3a}
Let $f \in (0, \frac{1}{3})$ be real. The \sagefunc{rlm_dpl1_extreme_3a} function $\pi$ defined as follows is extreme:
\[
\pi(x) =
  \begin{cases}
   \frac{1}{f} x & \text{if } 0 \leq x \leq f \\
   \frac{2}{1+2f} x & \text{if } f < x < \frac{1+f}{2} \\
   \frac{1}{2} & \text{if } x = \frac{1+f}{2} \\
   \frac{2}{1+2f} x - \frac{1}{1+2f}  & \text{if } \frac{1+f}{2} < x < 1 \\
   0  & \text{if } x = 1
  \end{cases}
\]
\end{theorem}
\begin{proof}
Minimality of $\pi$ follows from~\cite[Definition~19, Proposition~18 and Corollary~23]{Richard-Li-Miller-2009:Approximate-Liftings}.



The proof of extremality 
follows the basic roadmap mentioned in~\cite[\autoref{survey:s:roadmap}]{igp_survey}. Suppose that
$\pi= \frac{1}{2}(\pi^1 + \pi^2)$, where $\pi^1,\pi^2$ are valid functions.  Define the \emph{additivity domain} of $\pi$ as
\begin{equation}
    E(\pi) := \setcond{(x, y)\in\R\times\R} { \pi(x) + \pi(y) -\pi(x + y) = 0}.
 \end{equation} 
By \cite[\autoref{survey:lemma:tight-implies-tight}]{igp_survey}, 
\begin{enumerate}[\rm(i)]
\item $\pi^1,\pi^2$ are minimal; 
\item all subadditivity relations $\pi(x + y) \le \pi(x) + \pi(y)$ that are tight for~$\pi$ are also tight for $\pi^1,\pi^2$, i.e., $E(\pi) \subseteq E(\pi^1), E(\pi^2)$;  
\item because $\pi$ is continuous from the right at $x=0$, the functions $\pi^1,\pi^2$ are continuous at all points at which $\pi$ is continuous.
\end{enumerate} 
Consider $\pi$, $\pi^1$, $\pi^2$ as solutions to the following system of linear equations:
\begin{equation}
\label{equation:system}
\begin{cases}
\varphi(0) = 0,\\
\varphi(f) = 1,\\
\varphi(1) = 0,\\
\varphi(u) + \varphi(v)  =  \varphi(u + v) 
,\quad (u,v) \in E(\pi),
\end{cases}
\end{equation}
where $\varphi$ is a minimal function that is continuous at all points at which $\pi$ is continuous.
We will show that \eqref{equation:system} has a unique solution. This will imply that $\pi = \pi^1 = \pi^2$, thereby establishing the extremality of $\pi$.

Following \cite[\autoref{survey:section:delta-p-definition}]{igp_survey}, we
define for any intervals $I,J,K\subseteq \R$ the set
\begin{equation}\label{eq:F-def}
F(I,J,K) := \setcond{(x,y) \in I \times J}{x + y \in K} \subseteq \R \times \R.
\end{equation}
Define the projections $p_1,p_2,p_3\colon \R\times \R \to \R$ by
\begin{equation}
\label{eq:projections}
p_1(x,y) = x, \quad p_2(x,y) = y, \quad  p_3(x,y) = x+y.
\end{equation}
Consider the closed triangle
\begin{align*}
  F_1 &= F\bigl(\, [0,f],\; [0,f],\; [0,f]\,\bigr) \\
  \intertext{and the open triangles}
  F_2 &= F\bigl(\,(b, 1),\; (b, 1), \; (1 + f, 1+b)\,\bigr), \\
  F_3 &= F\bigl(\,(f, b),\; (f, b), \; (2f, b)\,\bigr),
\end{align*}
where $b=\frac{1+f}{2}$. See \autoref{fig:rlm_dpl1_extreme_3a}
for an illustration.
Then, for any $(x, y) \in F_1$, \[\pi(x)+\pi(y)-\pi(x+y) = \frac{x}{f}  + \frac{y}{f}  - \frac{x+y}{f} =0.\] For any $(x, y) \in F_2$, \[ \pi(x)+\pi(y)-\pi(x+y) = \frac{2x -1}{1+2f} + \frac{2y-1}{1+2f}- \frac{2(x+y-1)}{1+2f} =0.\]
For any $(x, y) \in F_3$, \[ \pi(x)+\pi(y)-\pi(x+y) = \frac{2x}{1+2f} +
  \frac{2y}{1+2f} - \frac{2(x+y)}{1+2f} = 0.\] 
Hence $F_1, F_2, F_3 \subset E(\pi)$.
By the convex additivity domain lemma
\cite[\autoref{survey:lem:projection_interval_lemma_fulldim}]{igp_survey}
applied to the full-dimensional convex set $F_1$,
$\varphi$ is affine on the open interval $\intr(p_1(F_1)) = (0,f)$.
(We say that the interval $(0,f)$ is \emph{covered}.) Then $\varphi$ is uniquely determined on $[0,f]$ since $\varphi(0)=0$, $ \varphi(f) = 1$, and $\varphi$ is continuous from the right at $0$ and from the left at $f$. 

Similarly, by
\cite[\autoref{survey:lem:projection_interval_lemma_fulldim}]{igp_survey}
applied to $F_2$, $\varphi$ is affine on $(b, 1)$. Let $s$ be the slope. By symmetry, $\varphi$ is also affine on $(f, b)$ with the same slope $s$. 
Let $\varphi(u^-)$ and $\varphi(u^+)$ denote the left limit and right limit
value at $u$, respectively. By the symmetry condition of minimal valid
functions \cite[\autoref{survey:thm:minimal}]{igp_survey}, $\varphi(f^+) + \varphi(1^-) =1$ and $\varphi(b^-) + \varphi(b^+) = 1$. Let $t = \varphi(f^+)$, and thus $\varphi(b^-) = t + \frac{1-f}{2}s$, $\varphi(b^+) = 1 - t - \frac{1-f}{2}s$, and $\varphi(1^-) = 1-t$. Also note that $\varphi(1) = 0$ and $\varphi(b)=\frac{1}{2}$. Therefore $\varphi$ is uniquely determined by $s$ and $t$.

Now set up two more equations using the additivity relations:
\begin{align}\label{add_eq_1}
  \varphi(f^+) + \varphi(f^+) &= \varphi(2f^+),\\
\intertext{which corresponds to the lower left corner of the open triangle
  $F_3$, and}
  \label{add_eq_2}
  \varphi(b^+) + \varphi(b^+ ) &= \varphi(f^+),
\end{align}
which corresponds to the lower left corner of the open triangle $F_2$.
Since $0 < f < 1/3$, we have $f < 2f < b$. It follows that $\varphi(2f^+) = t + fs$. Hence $2t=t+fs$ by \eqref{add_eq_1} and $2(1-t-\frac{1-f}{2}s)=t$ by \eqref{add_eq_2}. The system has a unique solution, namely $s=\frac{2}{1+2f}$ and $t=\frac{2f}{1+2f}$. This shows that $\varphi$ is unique on $[0,1]$, completing the proof of extremality of $\pi$.  
\end{proof}              

\section{Failure of Chen's 3-slope construction}
\label{s:chen_3_slope_not_extreme}

\begin{figure}[t]
  \centering
  \includegraphics[width=0.45\linewidth]{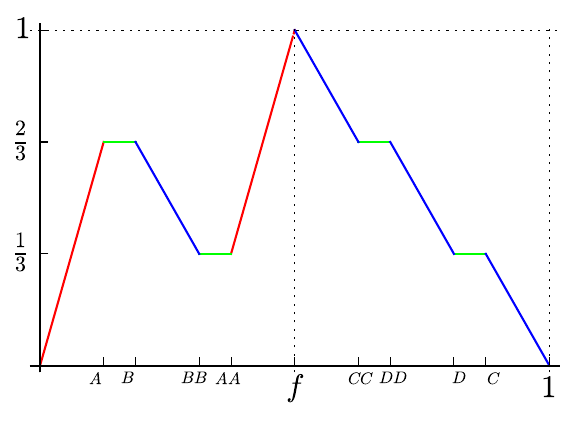}\quad
  \includegraphics[width=0.45\linewidth]{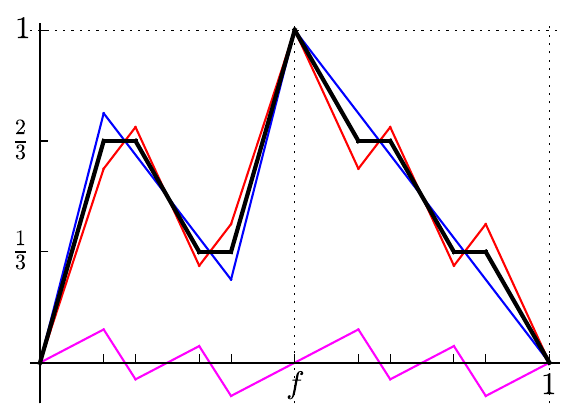}
  \caption{The function \sagefunc{chen_3_slope_not_extreme} is 
    minimal, but not extreme, as proved by
    \sage{\sagefunc{extremality_test}(h, show\underscore{}plots=True)}.
    The procedure first shows that
    for any distinct minimal $\pi^1 = \pi + \bar\pi$ (\emph{blue}), $\pi^2 = \pi
    - \bar\pi$ (\emph{red}) such that $\pi = \tfrac{1}{2}\pi^1
    + \tfrac{1}{2} \pi^2$, the functions $\pi^1$ and $\pi^2$ are continuous
    piecewise linear with the same breakpoints as $\pi$ (in the terminology of
    \cite{basu-hildebrand-koeppe:equivariant}, $\pi$ is \emph{affine imposing}
    on all intervals between breakpoints).  A finite-dimensional extremality
    test then finds a perturbation~$\bar\pi$ (\emph{magenta}), as shown.}
  \label{fig:chen_3_slope_not_extreme}
\end{figure}

Proofs of minimality and extremality of given piecewise linear functions
follow a standard pattern, but the complexity of these proofs grows with the
number of pieces.  Writing and verifying these proofs is a tedious and
error-prone task, which makes the case for computer-based proofs.  

We illustrate this by a family of 3-slope functions
(\autoref{fig:chen_3_slope_not_extreme}) constructed by Chen \cite{chen},
depending on the value $f \leq \frac{1}{2}$ and real parameters $A$, $B$, $C$, $D$, satisfying
certain conditions. Chen claims these functions to be minimal and extreme under these
conditions; however, his proofs are flawed.\footnote{The subadditivity checks
  in \cite[section~2.1 and 2.2]{chen} are insufficient, since the case
  $\pi(u_1) + \pi(u_2) \geq \pi(u_1+u_2)$ with $u_1$ and $u_1 + u_2$ being
  endpoints of $\pi$ is missed.  This flaw can be easily repaired; the functions
  are actually minimal. 
  However, the extremality proof is flawed and cannot
  be repaired.  In \cite[page~37]{chen}, the author
  claims that
  $-\pi(C ) + \pi(CC) + \pi(D) - \pi(DD) = 0$ 
  must hold because
  there is a ``height decrease'' of $1$ in the function from $\pi(f)$ and
  $\pi(1)$. In fact, this is false whenever
  the common slope value over the intervals $[A, B], [BB, AA], [CC, DD]$ and $[D,C]$,
  called $s_3$, is non-zero.
  Since one equation is missing from the system of six equations in six unknowns claimed by the author,  the system does not have a unique solution. As a result, $\pi$ is not extreme.}
The software \cite{infinite-group-relaxation-code}, as part of which the
Electronic Compendium is released, allows us to test the extremality for a
given function (i.e., for fixed parameters $f$, $A$, $B$, $C$, $D$).
Computationally we verified, by randomly generating thousands of functions in the described family, that none of these functions appear to be extreme.



\section{A standard proof of the extremality of
  \sage{drlm\underscore{}backward\underscore{}3\underscore{}slope} that
  removes the assumption of rational data}
\label{s:new_proof_drlm_backward_3_slope}

\begin{figure}[t]
  \centering
  \includegraphics[width=0.8\linewidth]{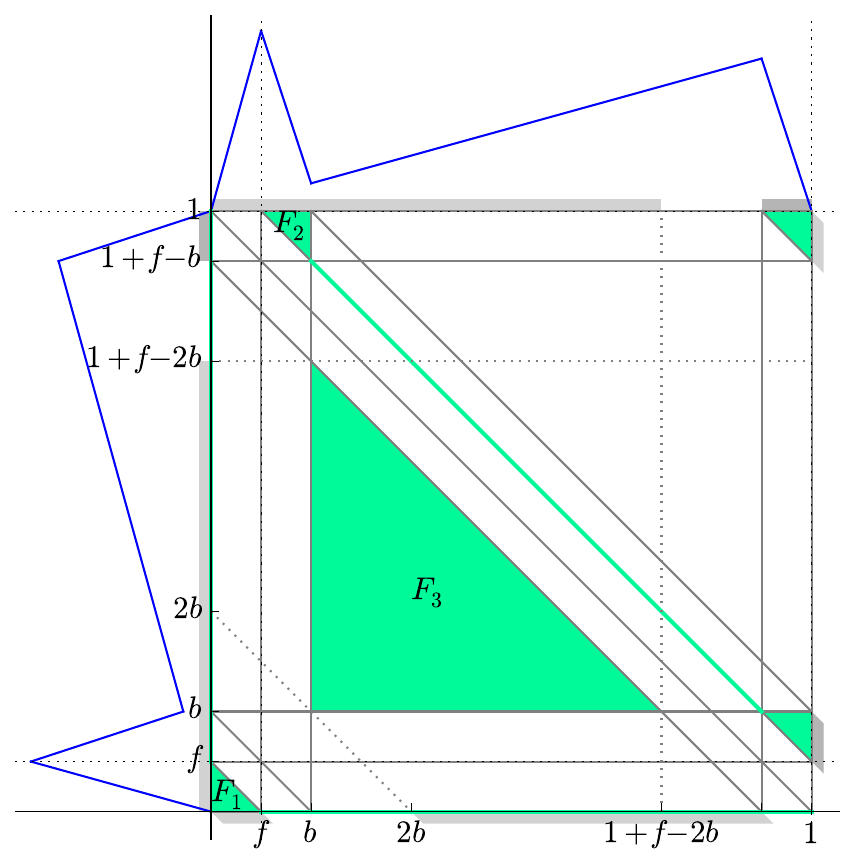}
  \caption{The \sagefunc{drlm_backward_3_slope} function}
  \label{fig:drlm_backward_3_slope}
\end{figure}

The function
\sagefunc{drlm_backward_3_slope} (\autoref{fig:drlm_backward_3_slope}) was discovered in \cite{dey1} and proven to be extreme for rational numbers $f$ and $b$ satisfying $f < b < (1+f)/4 < 1$. The proof is based on interpolation of extreme functions for finite cyclic groups \cite{AraozEvansGomoryJohnson03}, where the assumption of rational data is needed.

In fact, following the same roadmap as in the proof of \autoref{thm:rlm_dpl1_extreme_3a}, one can show that \sagefunc{drlm_backward_3_slope} is extreme 
without assuming $f$ and $b$ to be rational numbers.
 
\begin{theorem} \label{thm:drlm_backward_3_slope}
Let $f$ and $b$ be real numbers such that $0 < f < b \leq \frac{1+f}{4}$. The \sagefunc{drlm_backward_3_slope} function $\pi$ defined as follows is extreme:
\[
\pi(x) =
  \begin{cases}
   \frac{x}{f} & \text{if } 0 \leq x \leq f \\
   1 + \frac{(1+f-b)(x-f)}{(1+f)(f-b)} & \text{if } f \leq x \leq b \\
   \frac{x}{1+f} & \text{if } b \leq x \leq 1+f-b \\
   \frac{(1+f-b)(x-1)}{(1+f)(f-b)}  & \text{if } 1+f-b \leq x \leq 1
  \end{cases}
\]
\end{theorem}

\begin{proof}
First, it is straightforward to check that the function $\pi$ is symmetric and
subadditive. Thus, by \cite[\autoref{survey:thm:minimal}]{igp_survey}, $\pi$
is minimal. 
Second, by the convex additivity domain lemma for continuous functions \cite[\autoref{survey:lem:projection_interval_lemma-corollary}]{igp_survey}
applied to
\begin{align*}
  F_1 &= F\bigl(\, [0,f],\; [0,f],\; [0,f]\,\bigr), \\
  F_2 &= F\bigl(\, [f,b],\; [1+f-b,1],\; [1+f,1+b]\,\bigr), \\
  F_3 &= F\bigl(\, [b,1+f-2b],\; [b,1+f-2b],\; [2b,1+f-b]\,\bigr),
\end{align*}
the intervals $[0, f]$, $[f,b]$, $[b, 1+f-b]$, and $[1+f-b,1]$ are all covered
with slope values $s_1$, $s_2$, $s_3$, and $s_2$, respectively (see \autoref{fig:drlm_backward_3_slope}). Finally, set
up a system of equations using $\varphi(0) =\varphi(1) =0$, $\varphi(f)=1$ and
$\varphi(b)+\varphi(b) = \varphi(2b)$. This system has a unique
solution. Therefore, $\pi$ is an extreme function. 
\end{proof}

\section{A continuous ``1-slope function'' as a  limiting case of Basu et al.'s non--piecewise linear extreme function}
\label{s:bccz_counterexample}

Basu, Conforti, Cornu\'ejols and Zambelli~\cite{bccz08222222} constructed an
extreme non--piecewise linear function, disproving a conjecture of Gomory and
Johnson~\cite[section 6.1]{tspace}.  Their function is defined as a
limit of continuous piecewise linear functions which are defined as follows.
Let $f\in (0,1)$ be real.  
Consider a sequence of real numbers $\epsilon_1 > \epsilon_2 >
\dots$ such that $\epsilon_1 \leq 1-f$ and
\begin{equation} \label{series-less-than-1}
  \mu^- = (1-f) + \sum_{i=1}^{+\infty} 2^{i-1} \epsilon_i < 1.
\end{equation}
Basu et al.\ give a recursive construction of a function $\psi_n\colon\R\to\R$, each
depending on parameters $f$ and $\epsilon_1, \dots, \epsilon_n$; the function
$\psi_0$ is the \sagefunc{gmic} function \cite[section~3]{bccz08222222}.
We state a first observation without proof; see \autoref{fig:kf_n_step_mir}
for an illustration.
\begin{observation}
  Basu et al.'s functions $\psi_n$ are special cases of 
  \sagefunc{kf_n_step_mir} functions \cite{kianfar1} with the following parameters:
  \[f = f, \quad a_0 = 1, \quad a_1 = \frac{f + \epsilon_1}{2},\]
  \[a_i = \frac{a_{i-1} - \epsilon_{i-1} + \epsilon_i}{2},
  \text{ for } i = 2, 3, \dots, n.\]
\end{observation}
\begin{figure}[t]
  \centering
  \includegraphics[width=0.8\linewidth]{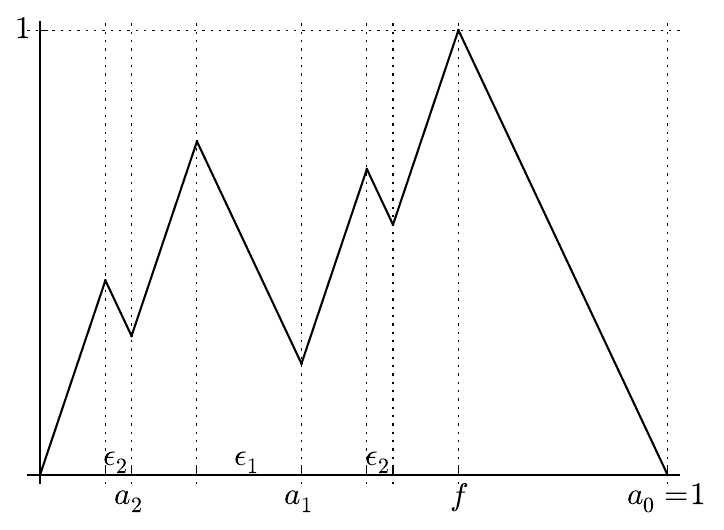}
  \caption{The \sagefunc{kf_n_step_mir} function}
  \label{fig:kf_n_step_mir}
\end{figure}
Basu et al.~\cite{bccz08222222} prove that the sequence formed by these
functions converges uniformly and hence gives a continuous limit
function~$\psi$.  The limit function is a 
fractal, an absolutely continuous, measurable, non--piecewise
linear ``2-slope function'' with the following properties:
\begin{enumerate}[\rm(i)]
\item There is a set $X^-\subseteq[0,1]$, a countable union of open intervals,
  of Lebesgue measure~$\mu^- < 1$, on which the function is differentiable
  with the derivative taking the same, negative value.
\item There is a set $X^+\subseteq[0,1]$ of measure~$\mu^+ = 1-\mu^- > 0$, which is nowhere
  dense, i.e., it does not contain any interval.  Such a set is sometimes
  called a \emph{fat Cantor set}.  The function is differentiable on $X^+$ as
  well, and the derivative takes the same, positive value on the points
  of~$X^+$.
\item There is an at most countable subset where the function is not differentiable.
\end{enumerate}

\medbreak

In this note, we consider an interesting fringe case of the same construction,
where $\mu^-$ in~\eqref{series-less-than-1} equals~$1$.  We restrict ourselves to the case 
of geometric sequences, which are the prime examples for the case $\mu^-<1$ as
well.

So consider the geometric sequence $\{\,\epsilon_i:i=1,2,\ldots\,\}$ with
common ratio $\frac{1}{q}$ and the first term $\epsilon_1 = \frac{q-2}{q} f$,
where $2 < q \leq \frac{2f}{2f-1}$ for $f > \frac{1}{2}$ and $2 < q$ for $f
\leq \frac{1}{2}$.  
Then $\{\epsilon_i\}$ is a decreasing sequence of positive reals, with
$\epsilon_1 \leq 1-f$ and
\begin{equation} \label{series}
  \mu^- = (1-f) + \sum_{i=1}^{+\infty} 2^{i-1} \epsilon_i = 1.
\end{equation}
Again let $\psi_i$ be the valid function defined in 
\cite[section~3]{bccz08222222}, for $i = 1,2,3, \ldots$. 
Let \[\gamma_i = f - \sum_{k=1}^i 2^{k-1}\epsilon_k = \left(\frac{2}{q}\right)^i f.\] 
By \eqref{series}, $\gamma_i > 0$ for every $i \geq 0$. Thus, by
\cite[Fact~4.1]{bccz08222222}, $\psi_i$ is well-defined.

We will now show by a simple calculation that in our case still uniform convergence holds.  The
following result replaces \cite[Lemma 5.2]{bccz08222222}.

\begin{lemma}
The sequence $\psi_1,\psi_2,\ldots$ is uniformly convergent.
\end{lemma}
\begin{proof}
By \cite[Fact~4.1]{bccz08222222}, there are $2^i$ intervals where $\psi_i$ has positive slope $s_i = \frac{1-\gamma_i}{(1-f)\gamma_i}$, each of length $\frac{\gamma_i}{2^i}$.
Note that $|\psi_i(x) - \psi_{i+1}(x)| \leq s_{i+1} \frac{\gamma_i}{2^i}$ since the values of the two functions match at the ends of the positive-slope intervals of $\psi_i$.
Since \[s_{i+1} \frac{\gamma_i}{2^i} =\frac{1-\gamma_{i+1}}{(1-f)\gamma_{i+1}}\cdot\frac{\gamma_i}{2^i} \leq \frac{1}{1-f} \frac{\gamma_i}{\gamma_{i+1}} \frac{1}{2^i} = \frac{q}{2(1-f)} \frac{1}{2^i},\]
we know that $|\psi_i(x) - \psi_{i+1}(x)| \leq C \frac{1}{2^i}$, where $C = \frac{q}{2(1-f)}$. Therefore, if $n<m$ then
\[ |\psi_n(x) - \psi_m(x)| \leq \sum_{i=n}^{m-1} C \frac{1}{2^i} \leq \sum_{i=n}^{\infty} C \frac{1}{2^i} = C \frac{1}{2^{n-1}}.\]

This implies that the sequence is Cauchy and hence convergent.
Thus the pointwise limit
\begin{equation}
\psi(x) = \lim_{i \rightarrow \infty} \psi_i(x)\label{PSI}
\end{equation}
of this sequence of functions is well defined.
Moreover, since the bound on $|\psi_n(x) - \psi_m(x)|$ does not depend on $x$, the above argument immediately implies that the sequence of functions $\psi_i$ converges uniformly to $\psi$.
\end{proof}
\begin{remark}
  When $i$ tends to $\infty$, $\gamma_i$ tends to $0$ and the positive slope
  $s_i$ tends to $\infty$.  Thus the convergence fails in the sense of
  $W^{1,1}_{\mathrm{loc}}(\R)$, i.e., the limit function is no longer
  absolutely continuous.
\end{remark}

The limit function is a minimal valid function because 
limits preserve minimality \cite[\autoref{survey:s:limits}]{igp_survey}.
Since, in general, extremality is not preserved by limits (see again
\cite[\autoref{survey:s:limits}]{igp_survey} for a discussion), the
extremality of the limit function requires a proof that uses the detailed
structure of the functions~$\psi_n$.  The proof in~\cite{bccz08222222} extends
verbatim to our case and will not be repeated here.\smallbreak

Hence $\psi$ is a continuous (but not absolutely continuous), measurable,
non--piecewise linear, extreme ``1-slope function.''  It is available in the
Electronic Compendium as a special case of \sagefunc{bccz_counterexample}.
\begin{remark}
  The function constructed by \sagefunc{bccz_counterexample} 
  can be exactly evaluated only on the union of the closures of the open
  intervals that form the set $X^-$; for other values, the function
  will return an approximation.  This enables simple computational experiments
  with this function, but the automated minimality and extremality test is not
  available.  
\end{remark}
\begin{openquestion}
  It is an open question if computational methods can be developed that can
  assist with the discovery, construction, and extremality test of
  non--piecewise linear functions that are like \sagefunc{bccz_counterexample}
  defined as limits of piecewise linear functions.
\end{openquestion}
\section{The Electronic Compendium}
\label{s:usage}

The Electronic Compendium is implemented in Python within the framework of the
open-source computer algebra system Sage \cite{sage}.  It is available as part
of the software~\cite{infinite-group-relaxation-code}.  It can be run on a
local installation of Sage, or online via \emph{SageMathCloud}.

An overview of the available functions is given
in \autoref{tab:compendium}.  These functions are defined in the Sage files
\path{extreme_functions_in_literature.sage} and \path{survey_examples.sage}.

\autoref{tab:sage-session} shows a sample Sage session, illustrating the basic use
of the software and the help system.  The latter provides a discussion of
parameters of the extreme functions, 
bibliographic information, etc.  It is accessed by typing a function name,
followed by a question mark.  

Following the standard conventions of Sage, the documentation strings contain
usage examples with their expected output.  If Sage is invoked as \texttt{sage -t
  $\langle\mathit{filename}\rangle$.sage}, these examples are run and
their results are compared to the expected results; if the results differ,
this is reported as a unit test failure.  This helps to ensure the
consistency and correctness of the compendium as we continue to add newly
discovered functions or generalize their constructions.

\begin{table}
  \caption{A sample Sage session}
  \label{tab:sage-session}
  \tiny
  \begin{tabular}{@{}p{\linewidth}@{}}
    \toprule
\begin{verbatim}
## First load the code.
sage: import igp; from igp import *
## Load a function and store it in variable h.
sage: h = gmic()
INFO: 2014-11-18 17:20:34,863 Rational case.
## Plot the function.
sage: plot_with_colored_slopes(h)
## The documentation string of each function reveals its optional arguments, 
## usage examples, and bibliographic information. 
sage: gmic?
[...]
Definition: gmic(f=4/5)
Docstring:
   Summary:
      * Name: GMIC (Gomory mixed integer cut);

      * Infinite (or Finite); Dim = 1; Slopes = 2; Continuous;
        Analysis of subadditive polytope method;

      * Discovered [55] p.7-8, Eq.8;

      * Proven extreme (for infinite group) [60] p.377, thm.3.3;
        (finite group) [57] p.514, Appendix 3.

      * (Although only extremality has been established in literature,
        the same proof shows that) gmic is a facet.

   Parameters:
      f (real) in (0,1).

   Examples:
      [61] p.343, Fig. 1, Example 1

         sage: logging.disable(logging.INFO)             # Suppress output in automatic tests.
         sage: h = gmic(4/5)
         sage: extremality_test(h, False)
         True

   Reference:
      [55]: R.E. Gomory, An algorithm for the mixed integer problem,
      Tech. Report RM-2597, RAND Corporation, 1960.

      [...]
## The docstring tells us that we can set the `f' value using an optional argument.
sage: h = gmic(1/5)
INFO: 2014-11-18 17:21:12,172 Rational case.
sage: extremality_test(h, show_plots=True)
INFO: 2014-11-18 17:21:12,177 pi(0) = 0
INFO: 2014-11-18 17:21:12,178 pi is subadditive.
INFO: 2014-11-18 17:21:12,178 pi is symmetric.
INFO: 2014-11-18 17:21:12,178 Thus pi is minimal.
INFO: 2014-11-18 17:21:12,178 Plotting 2d diagram...
INFO: 2014-11-18 17:21:12,178 Computing maximal additive faces...
INFO: 2014-11-18 17:21:12,179 Computing maximal additive faces... done
INFO: 2014-11-18 17:21:12,543 Plotting 2d diagram... done
INFO: 2014-11-18 17:21:12,543 Computing covered intervals...
INFO: 2014-11-18 17:21:12,544 Computing covered intervals... done
INFO: 2014-11-18 17:21:12,544 Plotting covered intervals...
INFO: 2014-11-18 17:21:12,776 Plotting covered intervals... done
INFO: 2014-11-18 17:21:12,776 All intervals are covered (or connected-to-covered). 2 components.
INFO: 2014-11-18 17:21:12,784 Finite dimensional test: Solution space has dimension 0
INFO: 2014-11-18 17:21:12,784 Thus the function is extreme.
True
\end{verbatim}
    \\
    \bottomrule
  \end{tabular}
\end{table}

\clearpage 

\begin{table}
  \caption{An overview of the Electronic Compendium of extreme functions, available at
    \url{https://github.com/mkoeppe/infinite-group-relaxation-code}}
  \label{tab:compendium}
  \centering
  \begin{tabular}{@{}*5{p{.18\linewidth}}@{}}
    \toprule
    \CompendiumGridEntry{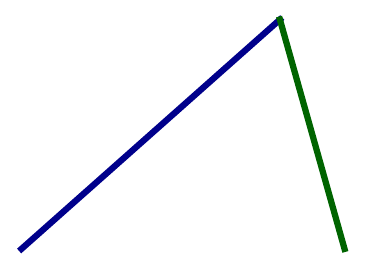}
    &\CompendiumGridEntry{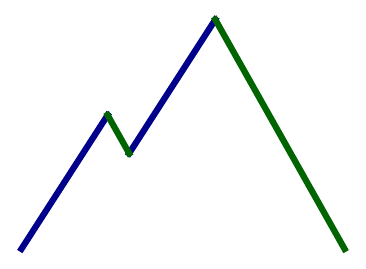}
    &\CompendiumGridEntry{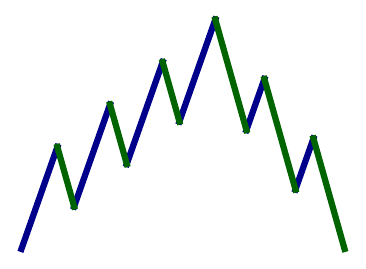}
    &\CompendiumGridEntry{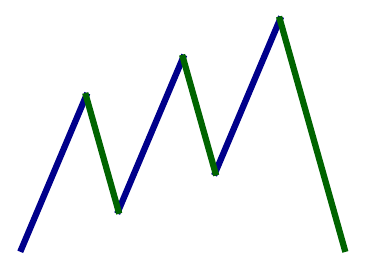}
    &\CompendiumGridEntry{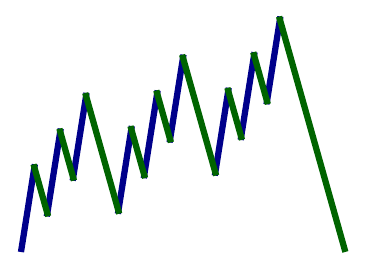}
    \\
    \CompendiumGridEntry{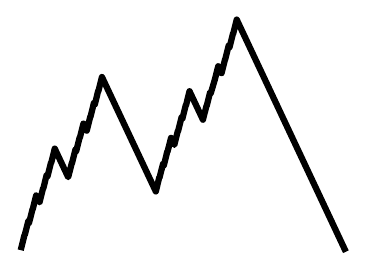}
    &\CompendiumGridEntry{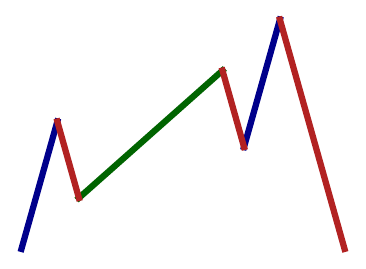}
    &\CompendiumGridEntry{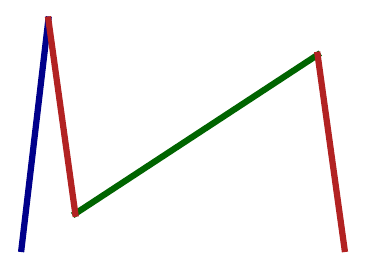}
    &\CompendiumGridEntry{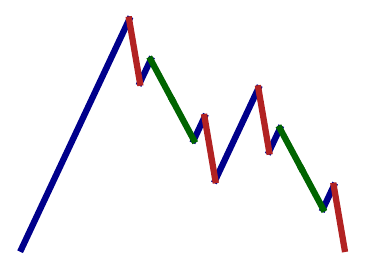}
    &\CompendiumGridEntry{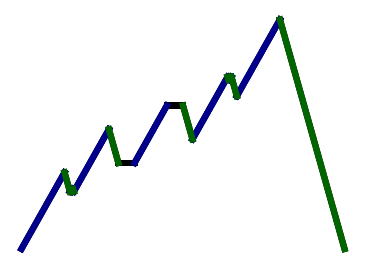}
    \\
    \CompendiumGridEntry{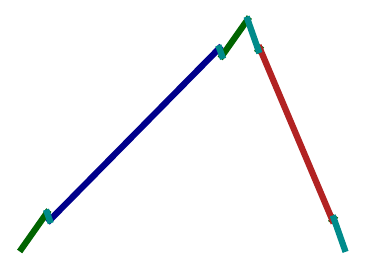}
    &\CompendiumGridEntry{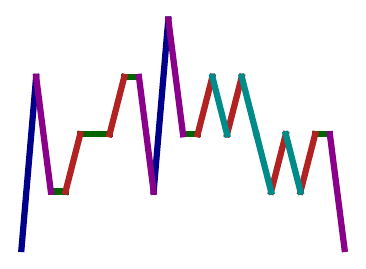}
    &\CompendiumGridEntry{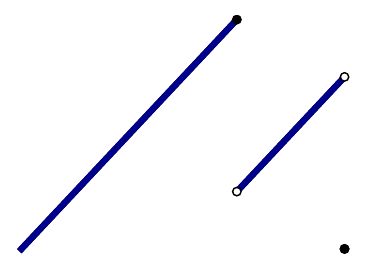}
    &\CompendiumGridEntry{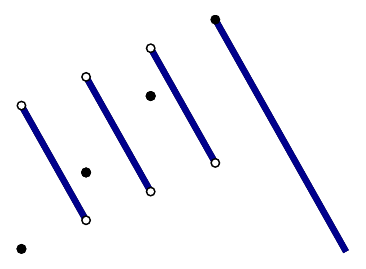}
    &\CompendiumGridEntry{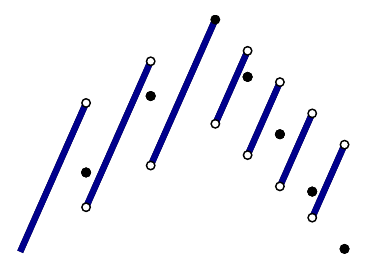}
    \\
    \CompendiumGridEntry{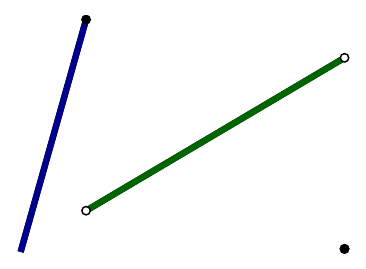}
    &\CompendiumGridEntry{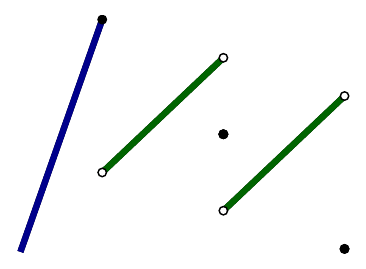}
    &\CompendiumGridEntry{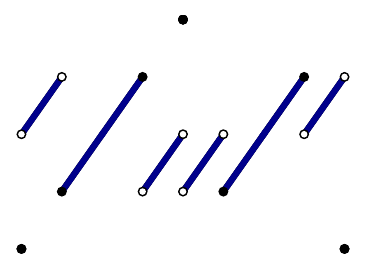}
    &\CompendiumGridEntry{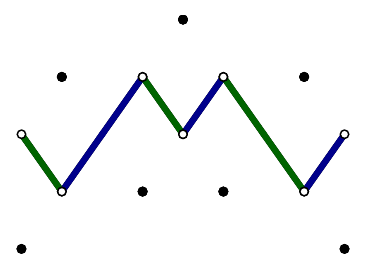}
    &\CompendiumGridEntry{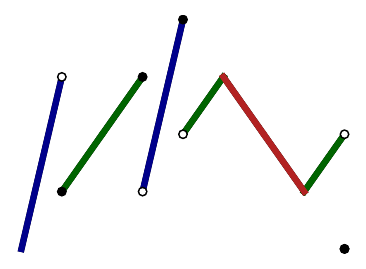}
    \\
    \bottomrule
  \end{tabular}
\end{table}


\providecommand\ISBN{ISBN }
\providecommand\CheckAccent[1]{\accent20 #1}
\providecommand{\bysame}{\leavevmode\hbox to3em{\hrulefill}\thinspace}
\providecommand{\MR}{\relax\ifhmode\unskip\space\fi MR }
\providecommand{\MRhref}[2]{%
  \href{http://www.ams.org/mathscinet-getitem?mr=#1}{#2}
}
\providecommand{\href}[2]{#2}

\end{document}